\documentclass{amsart}

\usepackage{amsfonts}
\usepackage{amsmath,amsthm}
\usepackage{amssymb}
\usepackage{geometry}

\usepackage[vcentermath]{youngtab}

\usepackage{hyperref}

\theoremstyle{plain}
\newtheorem{thm}{Theorem}[section]
\newtheorem{prop}[thm]{Proposition}
\newtheorem{lemma}[thm]{Lemma}

\theoremstyle{definition}

\newcommand{\EE}{\mathcal{E}}
\newcommand{\FF}{\mathcal{F}}
\newcommand{\X}{\mathcal{X}}

\newcommand{\R}{\mathbf{R}}
\newcommand{\Z}{\mathbf{Z}}

\newcommand{\1}{\mathbf{1}}
\newcommand{\TV}{\mathrm{TV}}
\newcommand{\tmix}{t_\mathrm{mix}}
\newcommand{\trel}{t_\mathrm{rel}}
\newcommand{\pimin}{\pi_\mathrm{min}}
\newcommand{\e}{\varepsilon}
\newcommand{\dsep}{d_{\textrm{sep}}}
\newcommand{\RRN}{\mathcal{R}_N}
\DeclareMathOperator{\Prob}{\mathbf{P}}

\DeclareMathOperator{\Var}{\textrm{Var}}
\DeclareMathOperator{\Binom}{\textrm{Binom}}

\DeclareMathOperator{\Real}{\textrm{Re}}
\DeclareMathOperator{\diam}{\textrm{diam}}
\newcommand{\cond}{\,|\,}
\newcommand{\st}{\,:\,}

\numberwithin{equation}{section}

\author{Daniel Jerison}
\title[General mixing time bounds]{General mixing time bounds for finite Markov chains via the absolute spectral gap}
\date{\today}

\begin{document}

\begin{abstract}
We prove an upper bound on the total variation mixing time of a finite Markov chain in terms of the absolute spectral gap and the number of elements in the state space. Unlike results requiring reversibility or irreducibility, this bound is finite whenever the chain converges. The dependence on the number of elements means that the new bound cannot capture the behavior of rapidly mixing chains; but an example shows that such dependence is necessary. We also provide a sharpened result for reversible chains that are not necessarily irreducible. The proof of the general bound exploits a connection between linear recurrence relations and Schur functions due to Hou and Mu, while the sharpened bound for reversible chains arises as a consequence of the stochastic interpretation of eigenvalues developed by Diaconis, Fill, and Miclo. In particular, for every reversible chain with nonnegative eigenvalues $\beta_j$ we find a strong stationary time whose law is the sum of independent geometric random variables having mean $1/(1-\beta_j)$.
\end{abstract}

\maketitle

\section{Introduction}

\subsection{Main results and discussion}

Let $P$ be the transition matrix for a discrete time Markov chain on a finite state space $\X$ of size $N$. Denote the eigenvalues of $P$ by $\beta_1,\ldots,\beta_N$, with $\beta_N = 1$. The absolute spectral gap of the chain is $1-\beta_\star$, where $\beta_\star = \max\{|\beta_j| \st 1\leq j\leq N-1 \}$.

The chain converges to a unique stationary distribution $\pi$ if and only if $\beta_\star < 1$. In this note we explore the quantitative relationship between the size of the absolute spectral gap and the time to stationarity. One of the first results in the modern theory of finite Markov chains, Proposition \ref{l2-basic} below, establishes this relationship when the chain is reversible and the value $\pimin = \min\{\pi(x) \st x\in\X\}$ is reasonably far from zero. The main contribution of the present work is an upper bound on total variation mixing time in terms of the absolute spectral gap, Theorem \ref{main}, that applies to all finite chains.

Suppose $\beta_\star < 1$. The total variation mixing time with parameter $0<\e<1$ is defined by
\[
\tmix(\e) = \min\{t\geq 0 \st \|P^t(x,\cdot) - \pi\|_\TV \leq \e \text{ for all } x\in\X \},
\]
where the total variation distance between two probability measures $\mu$ and $\nu$ on $\X$ is
\[
\|\mu - \nu\|_\TV = \sup_{S\subseteq\X} |\mu(S) - \nu(S)| = \frac{1}{2} \sum_{x\in\X} |\mu(x) - \nu(x)|.
\]
The relaxation time of the chain is $\trel = 1/(1-\beta_\star)$. We say $P$ is reversible (with respect to $\pi$) if  $\pi(x)P(x,y) = \pi(y)P(y,x)$ for all $x,y\in\X$. If $P$ is reversible and $\pimin>0$, all the eigenvalues $\beta_j$ are real. In that case we list them in increasing order, $-1 \leq \beta_1 \leq \beta_2 \leq \cdots \leq \beta_N = 1$.

\begin{prop} \label{l2-basic}
Let $P$ be a Markov transition matrix, reversible with respect to its stationary distribution $\pi$. For every $0<\e<1$,
\[
\log(2\e)^{-1} (\trel - 1) \leq \tmix(\e) \leq \left\lceil \frac{1}{2}(\log\pimin^{-1})\trel + \log(2\e)^{-1}\trel \right\rceil.
\]
\end{prop}

Proposition \ref{l2-basic} is originally due to Aldous, who proved a continuous time version in \cite{A82}. The lower bound given here is Theorem 12.4 in \cite{LPW09}. The upper bound follows from Proposition 3 in \cite{DS91}.

The lower bound in Proposition \ref{l2-basic} is sharp in many examples and holds even if $P$ is not reversible. (The proof given in \cite{LPW09} does not require reversibility.) The upper bound, while also sharp in some examples, does not extend to the nonreversible case. It can be poor when $\pimin$ is very small, and it gives no information when $\pimin = 0$ (such as when the Markov chain is absorbing).

A technical note: when we allow $\pimin = 0$, the class of reversible chains is not well-behaved. For example, they can have complex eigenvalues. To state our results precisely, we let $\RRN$ denote the set of $N\times N$ Markov transition matrices reversible with respect to a stationary distribution $\pi$ such that $\pimin>0$, and $\overline{\RRN}$ denote the closure of $\RRN$ in $\R^{N^2}$. Our later results requiring reversibility will apply to matrices in $\overline{\RRN}$.

The main theorem of this note needs no such assumptions.

\begin{thm} \label{main}
Let $P$ be a Markov transition matrix on a state space of size $N$. For every $0<\e<1$,
\[
\tmix(\e) \leq 2N\trel\log\trel + 4(1+\log 2)N\trel + 2(\log\e^{-1} - 1)\trel.
\]
\end{thm}

The bound in Theorem \ref{main} has a factor of $N$ instead of a factor of $\log \pimin^{-1}$. For chains with stationary distribution close to uniform, the upper bound in Proposition \ref{l2-basic} is of order $(\log N)\trel$, a significant improvement. Example \ref{skip-free} is a nonreversible chain with uniform stationary distribution whose mixing time is of order $N\trel$. Hence the replacement of $\log N$ with $N$ in Theorem \ref{main} is a necessary consequence of dropping the reversibility assumption.

Following Aldous \cite{A83}, we say a Markov chain mixes rapidly if its total variation mixing time is small compared to the size of its state space. Example \ref{skip-free} shows that one cannot conclude a nonreversible chain mixes rapidly based solely on upper bounds for the eigenvalues. So, necessarily, Theorem \ref{main} cannot capture the behavior of rapidly mixing chains.

The author has not been able to find a chain whose mixing time is slower than order $N\trel$. So it is an open question whether the extra factor of $\log\trel$ in Theorem \ref{main} is necessary or whether it is simply an artifact of the proof. But if we assume $P$ is reversible, the $\log\trel$ term can be eliminated.

\begin{thm} \label{rev-sharpen}
Let $P\in\overline{\RRN}$ be a reversible Markov transition matrix on a state space of size $N$. For every $0<\e<1$,
\[
\tmix(\e) \leq 2\left( N + 2\log\e^{-1} - 1 + \sqrt{2(N-2)\log\e^{-1}} \right) \trel.
\]
If the eigenvalues of $P$ are all nonnegative, the initial factor of $2$ can be omitted.
\end{thm}

For chains where $\pimin$ is exponentially small in $N$, the bounds given by Proposition \ref{l2-basic} and Theorem \ref{rev-sharpen} are comparable. If $\pimin$ is larger, the bound in Proposition \ref{l2-basic} is better. If $\pimin$ is super-exponentially small or zero, the bound in Theorem \ref{rev-sharpen} is better.

In section \ref{biased-walk} is a reversible chain for which $\pimin$ is exponentially small in $N$. Here the upper bounds from Proposition \ref{l2-basic} and Theorem \ref{rev-sharpen} are comparable, and both are sharp up to a constant factor. In the same section is a reversible chain for which $\pimin$ is super-exponentially small in $N$. The Proposition \ref{l2-basic} upper bound is no longer sharp, but the Theorem \ref{rev-sharpen} bound still is. Example \ref{pure-birth-section} is a reversible chain for which $\pimin = 0$. Here Proposition \ref{l2-basic} gives no upper bound, but Theorem \ref{rev-sharpen} is sharp.

Many chains mix significantly faster than the upper bounds in Theorems \ref{main}-\ref{rev-sharpen}. 
Example \ref{hypercube} is a rapidly mixing reversible chain for which neither the upper bound nor the lower bound in Proposition \ref{l2-basic} is sharp, but both are reasonably close, while Theorems \ref{main} and \ref{rev-sharpen} are far too conservative.

The following result is used to prove Theorem \ref{rev-sharpen} and is perhaps of independent interest. It is more powerful because it incorporates information about the entire spectrum rather than just the absolute spectral gap. In addition, it provides an upper bound on separation distance, which is stronger than total variation distance. The separation distance between two probability measures $\mu$ and $\nu$ on $\X$ is
\[
\dsep(\mu,\nu) = \max_{x\in\X} \left[ 1 - \frac{\mu(x)}{\nu(x)} \right].
\]
We have $\|\mu-\nu\|_\TV \leq \dsep(\mu,\nu)$.

\begin{thm} \label{sst}
Let $P$ be a Markov transition matrix on a state space $\X$ of size $N$. Suppose $P\in\overline{\RRN}$ is reversible with respect to $\pi$ and its eigenvalues $0 \leq \beta_1 \leq \cdots \leq \beta_{N-1} < \beta_N = 1$ are nonnegative. For $1\leq j\leq N-1$, let $\tau_j$ be independent geometric random variables with mean $1/(1-\beta_j)$. That is, for $t\geq 1$, $\Prob(\tau_j = t) = \beta_j^{t-1}(1-\beta_j)$. Then for any probability distribution $\mu$ on $\X$ and any $t\geq 0$,
\[
\dsep(P^t(\mu,\cdot),\pi) \leq \Prob(\tau_1 + \cdots + \tau_{N-1} > t).
\]
\end{thm}

Proving Theorem \ref{rev-sharpen} from Theorem \ref{sst} is relatively straightforward when the eigenvalues of $P$ are nonnegative. When $P$ has negative eigenvalues, we apply Theorem \ref{sst} to $P^2$ and use that the separation distance $\dsep(P^t(\mu,\cdot),\pi)$ is nonincreasing in $t$.

Example \ref{sticky-walk} is a reversible chain with nonnegative eigenvalues for which the lower bound in Proposition \ref{l2-basic} is sharp. Theorem \ref{sst} gives the correct matching upper bound. The upper bound from Proposition \ref{l2-basic} is slightly too high, and the upper bound from Theorem \ref{rev-sharpen} is somewhat higher.

Throughout this note we use the convention of writing measures as row vectors and functions as column vectors. If $\mu$ is a measure on $\X$ and $f$ is a function on $\X$, we have
\[
\mu P = P(\mu,\cdot) = \sum_{x\in\X} \mu(x)P(x,\cdot), \qquad \mu f = \sum_{x\in\X} \mu(x)f(x), \qquad Pf(x) = \sum_{y\in\X} P(x,y)f(y).
\]

\subsection{Methods of proof}

The proofs of Theorem \ref{main} and Theorem \ref{sst} are quite different. Theorem \ref{main} is proved using combinatorial arguments. Let $\Pi$ be the $N\times N$ matrix given by $\Pi(x,y) = \pi(y)$. The sequence $(P^t - \Pi: t\geq 0)$ satisfies a linear recurrence relation whose characteristic polynomial is $g(x) = (x-\beta_1)\cdots (x-\beta_{N-1})$. It follows that for each $t\geq 0$ there are coefficients $c_{t,k}$ such that
\[
P^t - \Pi = \sum_{k=0}^{N-2} c_{t,k} (P^k - \Pi).
\]
Let $C$ be the companion matrix for $g(x)$. This is an $N \times N$ matrix with ones below the diagonal, the coefficients of $-g(x)$ in the right column, and zeros everywhere else. Then the coefficients $c_{t,k}$ are entries of $C^t$.

In Section 4 we show that each entry of $C^t$ is a Schur function in the roots $\beta_1,\ldots,\beta_{N-1}$. This result is originally due to Hou and Mu \cite{HM03}, who obtain it as a special case of a more general theory. We provide a direct proof.

To finish proving Theorem \ref{main}, we use the Schur function formula to bound each coefficient $c_{t,k}$ in terms of $\beta_\star = \max\{|\beta_j| \st 1\leq j\leq N-1 \}$. The triangle inequality then gives a bound on $P^t - \Pi$.

Theorem \ref{sst} arises as a consequence of the stochastic interpretation of eigenvalues developed by Diaconis, Fill, and Miclo \cite{DM09} \cite{F09a} \cite{F09b} \cite{M10}. First, a strong stationary time for a Markov chain $(X_t: t\geq 0)$ having stationary distribution $\pi$ is a randomized stopping time $\tau$ such that for all $t\geq 0$ and $x\in\X$,
\[
\Prob(X_\tau=x \cond \tau=t) = \pi(x).
\]
Given a strong stationary time $\tau$ for $(X_t)$, we have for all $t\geq 0$ that
\[
\dsep(\Prob(X_t\in\cdot),\pi) \leq \Prob(\tau > t).
\]

Suppose $P\in\overline{\RRN}$ is reversible with respect to $\pi$ and all its eigenvalues are nonnegative. Define the pure birth transition matrix $Q$ on $\{1,\ldots,N\}$ by
\begin{equation} \label{pure-birth-matrix}
Q = \begin{bmatrix}
\beta_1 & 1-\beta_1 & 0 & 0 & \cdots & 0 & 0 \\
0 & \beta_2 & 1-\beta_2 & 0 & \cdots & 0 & 0 \\
0 & 0 & \beta_3 & 1-\beta_3 & \cdots & 0 & 0 \\
\vdots & \vdots & \vdots & \ddots & \ddots & \ddots & \vdots \\
0 & 0 & 0 & 0 & \cdots & \beta_{N-1} & 1-\beta_{N-1} \\
0 & 0 & 0 & 0 & \cdots & 0 & \beta_N
\end{bmatrix}.
\end{equation}
Recall that we list $\beta_1,\ldots,\beta_N$ in increasing order, so in particular $\beta_N=1$.

A link between $P$ and $Q$ is an $N\times N$ matrix $\Lambda$ whose rows are indexed by $\{1,\ldots,N\}$ and whose columns are indexed by $\X$, such that $\Lambda P = Q \Lambda$. It is a stochastic link if all the entries of $\Lambda$ are nonnegative and each row of $\Lambda$ sums to 1. Closely following arguments of Fill \cite{F09a} and Miclo \cite{M10}, in Section 3 we show that for any probability distribution $\mu$ on $\X$, there is a stochastic link $\Lambda$ between $P$ and $Q$ such that the first row of $\Lambda$ is $\mu$ and the $N$th row of $\Lambda$ is $\pi$.

Let $(X_t: t\geq 0)$ be the Markov chain on $\X$ with initial distribution $X_0 \sim \mu$ and transition matrix $P$, and let $(Y_t: t\geq 0)$ be the Markov chain on $\{1,\ldots,N\}$ with initial value $Y_0 = 1$ and transition matrix $Q$. In the language of Diaconis and Fill \cite{DF90}, $(Y_t)$ is a strong stationary dual for $(X_t)$. Thus, $(X_t)$ has a strong stationary time $\tau$ equal in law to $\min\{t\geq 0\st Y_t = N\}$, which is distributed as the sum of $N-1$ independent geometric random variables as in the statement of Theorem \ref{sst}.

\subsection{Other techniques}

In this subsection we discuss a few techniques that can be used in conjunction with or in place of Theorems \ref{main}-\ref{sst}. The list is far from exhaustive. For more information, see the books \cite{LPW09} and \cite{MT06}.

\subsubsection{Diagonalization}

Sometimes it is possible to find explicit formulas for the eigenvalues and eigenvectors of a Markov chain. A few examples are in Chapter 12 of \cite{LPW09}. If the chain is a random walk on a finite group, one can use representation theory to compute such formulas; see e.g. \cite{D88}.

Once the chain is diagonalized, Proposition \ref{l2-basic} or Theorems \ref{main}-\ref{sst} can be applied directly. Exact formulas for mixing time can be written in terms of the eigenvectors, but such formulas are sometimes difficult to work with. Still, this approach can lead to sharp bounds, as in Diaconis and Shahshahani's analysis of the random transposition walk on the symmetric group \cite{DS81}.

\subsubsection{Poincar\'e inequalities}

Suppose $P$ is reversible with respect to $\pi$ and $\pimin>0$. The variance of a function $f:\X\to\R$ with respect to $\pi$ is $\Var_\pi(f) = \pi(f^2) - (\pi f)^2$. By the variational characterization of eigenvalues (see Lemma 1.21 in \cite{MT06}),
\begin{align*}
1-\beta_{N-1} &= \inf_{\substack{f:\X\to\R \\ \Var_\pi(f)>0}} \frac{\EE(f,f)}{\Var_\pi(f)}, \\
\beta_1 - (-1) &= \inf_{\substack{f:\X\to\R \\ \Var_\pi(f)>0}} \frac{\FF(f,f)}{\Var_\pi(f)},
\end{align*}
where the Dirichlet form $\EE(f,f)$ and the parity form $\FF(f,f)$ are defined by
\begin{align*}
\EE(f,f) &= \frac{1}{2}\sum_{x,y\in\X} \pi(x)P(x,y)[f(x)-f(y)]^2, \\
\FF(f,f) &= \frac{1}{2}\sum_{x,y\in\X} \pi(x)P(x,y)[f(x)+f(y)]^2.
\end{align*}

A Poincar\'e inequality for the chain is a statement of the form
\begin{equation} \label{Poincare}
\Var_\pi(f) \leq \kappa \EE(f,f) \qquad \text{for all } f:\X\to\R.
\end{equation}
It follows immediately that $\beta_{N-1} \leq 1-\frac{1}{\kappa}$. Poincar\'e inequalities can be proved using path arguments, introduced by Jerrum and Sinclair \cite{JS89} and extended by Diaconis and Stroock \cite{DS91} and by Sinclair \cite{S92}. Choose for each pair $x,y\in\X$ a path $\Gamma_{xy} = (x=x_0,x_1,\ldots,x_d=y)$ with each $P(x_i,x_{i+1})>0$. Denote the length of such a path by $|\Gamma_{xy}| = d$. Sinclair \cite{S92} proved that given a choice of paths, \eqref{Poincare} holds with
\begin{equation} \label{canonical}
\kappa = \max_{\{(w,z):P(w,z)>0\}} \frac{1}{\pi(w)P(w,z)} \sum_{\substack{x,y\in\X \\ (w,z)\in\Gamma_{xy}}} \pi(x)\pi(y)|\Gamma_{xy}|.
\end{equation}
If all the paths $\Gamma_{xy}$ have odd length, it also holds that $\Var_\pi(f) \leq \kappa\FF(f,f)$ for all $f$, meaning that the absolute spectral gap is at least $\frac{1}{\kappa}$.

There are several variants of \eqref{canonical}, including one using a weighted average of paths between each pair $x,y\in\X$ instead of a single path. Diaconis and Saloff-Coste \cite{DSC93a} \cite{DSC93b} generalized the path method to a theory of comparison of Markov chains. Suppose $\hat{P}$ is another Markov transition matrix on $\X$, with Dirichlet form $\hat{\EE}$ and parity form $\hat{\FF}$. Instead of choosing a path (or weighted average of paths) between every pair $x,y\in\X$, one only chooses paths between pairs $x,y$ with $\hat{P}(x,y) > 0$. The result is a bound of the form $\hat{\EE}(f,f) \leq A\EE(f,f)$, so that a Poincar\'e inequality for $\hat{P}$ yields a Poincar\'e inequality for $P$. If all the paths have odd length, one also has $\hat{\FF}(f,f) \leq A\FF(f,f)$. Taking $\hat{P}(x,y) = \pi(y)$ for all $x,y\in\X$ recovers \eqref{canonical}. See \cite{DGJM06} for more details.

\subsubsection{Cheeger constants}

The Cheeger constant for the chain with transition matrix $P$ is given by
\[
\Phi = \min_{S\subseteq\X \st 0<\pi(S)\leq 1/2} \frac{\sum_{x\in S, y\notin S} \pi(x)P(x,y)}{\pi(S)}.
\]
Sinclair and Jerrum \cite{SJ89} and, independently, Lawler and Sokal \cite{LS88} proved that if $P$ is reversible and $\pimin>0$,
\begin{equation} \label{Cheeger}
\frac{\Phi^2}{2} \leq 1-\beta_{N-1} \leq 2\Phi.
\end{equation}
Bounds on $\Phi$ can be obtained using path arguments similar to those described above. Fulman and Wilmer \cite{FW99} show that in many cases, the resulting lower bound on the spectral gap $1-\beta_{N-1}$ is worse than the one obtained by a Poincar\'e inequality. They also point out instances where $\Phi$ can be bounded using other methods, and the resulting bound on the spectral gap is better than what is given by path arguments.

Combining the upper bound $1-\beta_{N-1} \leq 2\Phi$ with the lower bound of Proposition \ref{l2-basic} gives a lower bound on mixing time in terms of $\Phi$. For instance, one has
\[
\tmix\left(\frac{1}{2e}\right) \geq \frac{1}{2\Phi}-1.
\]
This bound also holds (up to a constant factor) without assuming reversibility. A direct argument not using eigenvalues gives (\cite{DGJM06}, Theorem 17; \cite{LPW09}, Theorem 7.3)
\[
\tmix\left(\frac{1}{2e}\right) \geq \frac{\frac{1}{2}-\frac{1}{2e}}{\Phi}.
\]

Suppose $P$ is reversible and lazy, that is, $P(x,x)\geq 1/2$ for all $x\in\X$. This means all the eigenvalues are nonnegative, so \eqref{Cheeger} bounds the absolute spectral gap from below. Combining with the upper bound of Proposition \ref{l2-basic}, we obtain
\[
\tmix(\e) \leq \left\lceil \frac{2}{\Phi^2} \log\left( \frac{1}{2\e\sqrt{\pimin}} \right) \right\rceil.
\]
The evolving sets process of Morris and Peres \cite{MP05} gives nearly the same bound without assuming reversibility. Specifically, (\cite{LPW09}, Theorem 17.10)
\begin{equation} \label{Cheeger-lazy}
\tmix(\e) \leq \left\lceil \frac{2}{\Phi^2} \log\left( \frac{1}{\e\pimin} \right) \right\rceil
\end{equation}
assuming only that $P$ is lazy.

\subsubsection{Reversibilization}

If $P$ is not reversible but we still have $\pimin>0$, define the time reversal transition matrix
\[
P^\ast(x,y) = \frac{\pi(y)}{\pi(x)} P(y,x), \qquad x,y\in\X.
\]
$P^\ast$ also has stationary distribution $\pi$. Two important reversible matrices that can be generated using $P$ and $P^\ast$ are the additive reversibilization $\frac{P+P^\ast}{2}$ and the multiplicative reversibilization $P^\ast P$.

Many of the arguments discussed previously for reversible chains extend naturally to the context of nonreversible chains using one or the other reversibilization. The Dirichlet form $\EE$ and parity form $\FF$ are the same for $P$ and $\frac{P+P^\ast}{2}$. Therefore, path and comparison arguments applied to $P$ can bound the second-highest eigenvalue $b_{N-1}$ of $\frac{P+P^\ast}{2}$ away from 1 and the lowest eigenvalue $b_1$ away from $-1$. The eigenvalues $\beta_1,\ldots,\beta_{N-1}$ of $P$ satisfy $b_1 \leq \Real \beta_j \leq b_{N-1}$ (\cite{MT06}, Theorem 5.10). Hence the Dirichlet form of $P$ controls the eigenvalue distance from 1, and the parity form controls the eigenvalue distance from $-1$, but absent other hypotheses we have no lower bound on the absolute spectral gap.

A simple example is the non-random walk on $\Z/3\Z$ given by $P(i,i+1) = 1$, which does not converge at all. One can use the Dirichlet and parity forms to bound the eigenvalues away from 1 and $-1$, but the absolute spectral gap is zero since the eigenvalues are the cube roots of unity. This example is discussed in \cite{DGJM06}, which has much more on comparison techniques for nonreversible chains.

The Cheeger constant $\Phi$ is also the same for $P$ and $\frac{P+P^\ast}{2}$. As above, a lower bound on $\Phi$ implies that the eigenvalues of $P$ cannot be too close to 1, but there is no control over the absolute spectral gap.

In contrast with the additive reversibilization, the multiplicative reversibilization $P^\ast P$ provides a lower bound for the absolute spectral gap $1-\beta_\star$ of $P$. Let $\alpha$ be the square root of the second-highest eigenvalue of $P^\ast P$. Then Fill \cite{F91} proves using an identity of Mihail \cite{M89} that $\beta_\star \leq \alpha$. One can now apply Theorem \ref{main} to bound the mixing time, but the following analogue of the Proposition \ref{l2-basic} upper bound will be much tighter as long as $\pimin$ is reasonably far from zero:
\begin{equation} \label{l2-ext}
\tmix(\e) \leq \left\lceil \frac{1}{1-\alpha} \log\left( \frac{1}{2\e\sqrt{\pimin}} \right) \right\rceil.
\end{equation}
This is exactly the same as Proposition \ref{l2-basic} with $1/(1-\alpha)$ in place of $\trel = 1/(1-\beta_\star)$. It follows from Theorem 2.1 of \cite{F91}.

While the bound \eqref{l2-ext} is encouraging, in many instances $P$ converges much more quickly than $P^\ast P$. Often $P^\ast P$ does not converge at all and $\alpha=1$. 

Suppose $P(x,x)\geq\delta$ for all $x\in\X$. Then the additive and multiplicative reversibilizations are related by (\cite{MT06}, Remark 1.16; see also \cite{F91}, Lemma 2.4)
\[
\EE_{P^\ast P}(f,f) \geq 2\delta\EE_P(f,f) = 2\delta\EE_{\frac{P+P^\ast}{2}}(f,f).
\]
In this way, a Poincar\'e inequality or Cheeger bound for $P$ leads to a Poincar\'e inequality for $P^\ast P$, which in turn bounds the mixing time of $P$ via \eqref{l2-ext}. For instance, if $\Phi$ is the Cheeger constant for $P$, the second-highest eigenvalue $b_{N-1}$ of $\frac{P+P^\ast}{2}$ satisfies $1-b_{N-1} \geq \Phi^2/2$. Suppose $\delta=1/2$. Then for all $f:\X\to\R$,
\[
\Var_\pi(f) \leq \frac{2}{\Phi^2}\EE_{\frac{P+P^\ast}{2}}(f,f) \leq \frac{2}{\Phi^2}\EE_{P^\ast P}(f,f),
\]
implying that $\alpha \leq 1-\Phi^2/4$. It follows from \eqref{l2-ext} that
\[
\tmix(\e) \leq \left\lceil \frac{4}{\Phi^2} \log\left( \frac{1}{2\e\sqrt{\pimin}} \right) \right\rceil.
\]
This is the same result (up to a constant factor) as \eqref{Cheeger-lazy}, with the same hypotheses on $P$.

\subsubsection{Coupling}

Coupling is a powerful technique for bounding Markov chain mixing times. In principle (and usually in practice) it is completely independent of eigenvalue methods, since the fundamental coupling inequality (\cite{LPW09}, Theorem 5.2) provides a direct relationship with mixing time. There is a connection, though, in that couplings satisfying a contraction property give bounds on the absolute spectral gap $1-\beta_\star$.

Let $\bar{P}$ be a Markov transition matrix on $\X\times\X$ with the property that
\[
\bar{P}((x,y),(x',\X)) = P(x,x'), \qquad \bar{P}((x,y),(\X,y')) = P(y,y').
\]
The chain on $\X\times\X$ with transition matrix $\bar{P}$ is a Markovian coupling of the original chain with itself. Suppose $\rho$ is a metric on $\X$, and there is $\theta<1$ such that
\[
\sum_{x',y'\in\X} \bar{P}((x,y),(x',y')) \rho(x',y') \leq \theta \rho(x,y) \qquad \text{for all } x,y\in\X.
\]
Then it follows (\cite{C98}; \cite{LPW09}, Theorem 13.1) that $\beta_\star \leq \theta$.

Suppose that $\rho(x,y)\geq 1$ for all $x\neq y$; this can always be achieved by scaling $\rho$ if necessary. If $\diam\X = \max\{\rho(x,y)\st x,y\in\X\}$, one can show using the fundamental coupling inequality that
\begin{equation} \label{couple-bound}
\tmix(\e) \leq \left\lceil \frac{1}{1-\theta} \log\left( \frac{\diam\X}{\e} \right) \right\rceil.
\end{equation}
This is similar to the upper bound in Proposition \ref{l2-basic}, except that reversibility is not required and the $\pimin^{-1}$ is replaced by $\diam\X$. In practice, $\diam\X$ is usually smaller than $\pimin^{-1}$, sometimes much smaller.

The method of path coupling, due to Bubley and Dyer \cite{BD97}, introduces a graph structure on the state space $\X$. One chooses a set $E$ of undirected edges $(x,y)$ such that the vertex set $\X$ is connected. Then one defines a function $\ell:E\to [1,\infty)$ giving the length of each edge. The path metric $\rho$ on $\X$ is given by
\[
\rho(x,y) = \min_{x_0,x_1,\ldots,x_d\in\X} \left\{ \sum_{i=1}^d \ell(x_{i-1},x_i) \st x_0=x,\ x_d=y,\ \text{each } (x_{i-1},x_i)\in E \right\}.
\]
If there is $\theta<1$ such that
\[
\sum_{x',y'\in\X} \bar{P}((x,y),(x',y')) \rho(x',y') \leq \theta \rho(x,y) \qquad \text{for all } (x,y)\in E,
\]
then (\cite{BD97}; \cite{LPW09}, Corollary 14.7) $\beta_\star \leq \theta$ and \eqref{couple-bound} holds. The advantage is that one need only define $\bar{P}$ starting at pairs of states $x,y$ adjacent in the graph.

\subsection{Organization}

The rest of this note is organized as follows. In Section 2 we discuss several examples illustrating the sharpness or lack thereof of Theorems \ref{main}-\ref{sst} in different situations. Section 3 introduces the stochastic interpretation of eigenvalues developed by Diaconis, Fill, and Miclo. A slight extension of this theory allows us to prove Theorem \ref{sst}, from which Theorem \ref{rev-sharpen} follows. In Section 4 we state and prove the combinatorial result of Hou and Mu about companion matrices and Schur functions. Section 5 uses the result of Section 4 to prove Theorem \ref{main}.

\section{Examples}

\subsection{Pure birth chain}
\label{pure-birth-section}

Set $\X = \{1,\ldots,N\}$ and fix $0<\beta<1$. Define $P$ by $P(i,i)=\beta$ and $P(i,i+1)=1-\beta$ for $1\leq i\leq N-1$, and $P(N,N)=1$. This is the transition matrix in \eqref{pure-birth-matrix} with $\beta_1 = \cdots = \beta_{N-1} = \beta$.

State $N$ is absorbing, so $\pimin = 0$. The eigenvalues are 1 and $N-1$ copies of $\beta$ in a Jordan block, meaning that $\beta_\star = \beta$. The mixing time is of order $N/(1-\beta) = N\trel$. Theorem \ref{rev-sharpen} is sharp; indeed, the theorem is proved by showing that this chain is the slowest mixing among all chains in $\overline{\RRN}$ with nonnegative eigenvalues and an absolute spectral gap of $1-\beta$.

\subsection{Biased random walk on the path}
\label{biased-walk}

Set $\X = \{1,\ldots,N\}$ and fix probabilities $p,q,r$ such that $p+q+r=1$. For $1\leq i\leq N-1$, let $P(i,i-1)=p$, $P(i,i)=q$, and $P(i,i+1)=r$. At the endpoints, let $P(0,0)=p+q$ and $P(0,1)=r$, while $P(N,N-1)=p$ and $P(N,N)=q+r$. So $p$ is the probability of moving left, $q$ is the probability of staying still, and $r$ is the probability of moving right. The stationary distribution is given by $\pi(i) = \frac{1}{Z}(r/p)^i$, where $Z$ is an easily computed normalizing constant. Assume $r>p$.

The eigenvalues are 1 and $q + 2\sqrt{pr}\cos(j\pi/N)$ for $1\leq j\leq N-1$. (See for example the computation in \cite{P09}.) Suppose $p,q,r$ are fixed as $N\to\infty$. Then $\trel$ is constant in $N$ while $\pimin$ is exponentially small. Hence Proposition \ref{l2-basic} and Theorems \ref{main}-\ref{rev-sharpen} all give mixing time upper bounds of order $N$, matching the actual mixing time up to a constant factor.

If instead we keep $q$ constant and send $p\to 0$ as $N\to\infty$, $\pimin$ is super-exponentially small. For instance, if $p=1/N$ then $\log\pimin^{-1}$ is of order $N\log N$. The relaxation time remains constant, and the mixing time is still of order $N$. Thus, Theorems \ref{main}-\ref{rev-sharpen} give the correct upper bound while Proposition \ref{l2-basic} is too conservative. This is essentially a perturbation of Example 2.1.

\subsection{Random walk on the path with sticky endpoint}
\label{sticky-walk}

Set $\X = \{1,\ldots,N\}$. For $1\leq i\leq N-1$, let $P(i,i-1) = 1/4$, $P(i,i) = 1/2$, and $P(i,i+1) = 1/4$. At the endpoints, let $P(0,0) = 3/4$ and $P(0,1) = 1/4$, while $P(N,N-1) = 1/4(N-1)$ and $P(N,N) = 1-1/4(N-1)$. In other words, we modify the unbiased lazy random walk on the path so that $P(N,N-1)$ is very small, that is, the right endpoint is ``sticky.''

The stationary distribution is $\pi(i) = 1/2(N-1)$ for $1\leq i\leq N-1$ and $\pi(N) = 1/2$. The mixing time and absolute spectral gap can be computed up to constant factors using formulas of Chen and Saloff-Coste \cite{CSC13}. Their Theorem 1.1 implies that the mixing time is order $N^2$, and their Theorem 1.2 implies that the relaxation time is also order $N^2$.

The mixing time upper bound from Proposition \ref{l2-basic} is order $N^2\log N$. Theorem \ref{rev-sharpen} gives an upper bound of $N^3$. We now show using arguments from Section 3 that the Theorem \ref{sst} upper bound is sharp to within a factor of 2 even for total variation distance.

Let $Q$ be given as in \eqref{pure-birth-matrix}. Define the measure $\mu$ on $\{1,\ldots,N\}$ by $\mu(1) = 1$ and $\mu(i) = 0$ for $i>1$. This is the same as the measure $\delta_1$ defined in the proof of Proposition \ref{link}, except that we view $\delta_1$ as a measure on the state space of $Q$ and $\mu$ as a measure on the state space of $P$.

By Proposition \ref{link}, there is a stochastic link $\Lambda$ between $P$ and $Q$ such that the first row of $\Lambda$ is $\mu$, the $N$th row of $\Lambda$ is $\pi$, and $\Lambda P = Q \Lambda$. It follows that $\Lambda P^t = Q^t \Lambda$ for all $t\geq 0$. Let $\mu_j = \Lambda(j,\cdot)$ as in the proof of Proposition \ref{link}. For all $t\geq 0$,
\begin{equation} \label{Q-decomp}
P^t(1,\cdot) = \mu P^t = \delta_1 \Lambda P^t = \delta_1 Q^t \Lambda = \sum_{j=1}^N Q^t(1,j)\mu_j.
\end{equation}
Since $\mu_N = \pi$, we have $P^t(1,\cdot) \geq Q^t(1,N)\pi$, so
\begin{equation} \label{sst-example}
\|P^t(1,\cdot)-\pi\|_\TV \leq \dsep(P^t(1,\cdot),\pi) \leq 1-Q^t(1,N).
\end{equation}
Indeed, if $\tau$ is the strong stationary time for $P$ constructed according to strong stationary duality as in \cite{DF90}, then $\Prob(\tau>t) = 1-Q^t(1,N)$, so \eqref{sst-example} is a restatement of the bound in Theorem \ref{sst}.

Because $P$ is skip-free, that is, $P(i,k)=0$ if $k>i+1$, it follows from \eqref{mu_j} that each $\mu_j$ is supported on $\{1,\ldots,j\}$. In particular, $\mu_j(N) = 0$ for $j<N$. Now \eqref{Q-decomp} implies that $P^t(1,N) = Q^t(1,N)\pi(N)$, and
\[
\pi(N)-P^t(1,N) = \pi(N)(1-Q^t(1,N)) \geq \pi(N) \|P^t(1,\cdot)-\pi\|_\TV \geq \pi(N)[\pi(N)-P^t(1,N)].
\]
The first and last terms differ by a multiple of $\pi(N) = 1/2$. Thus the middle inequality, which is Theorem \ref{sst}, is sharp to within a factor of 2.

The argument given here applies to any birth and death chain with nonnegative eigenvalues. It shows that Theorem \ref{sst} is off by at most a factor of $\pi(N)$. This result is somewhat interesting because the Theorem \ref{sst} bound depends on the entire spectrum of the transition matrix, yet we proved sharpness without needing to compute or bound any eigenvalues (besides ensuring nonnegativity).

\subsection{Skip-free chain with uniform stationary distribution}
\label{skip-free}

Set $\X = \{1,\ldots,N\}$. For $1\leq i\leq N-1$, let $P(i,j) = 1/i(i+1)$ for $j\leq i$ and $P(i,i+1) = i/(i+1)$. At the right endpoint, let $P(N,j) = 1/N$ for all $j$. The stationary distribution $\pi$ is uniform. The eigenvalues are 1 and $N-1$ copies of 0 in a Jordan block. If $\mu_j$ is the uniform distribution on $\{1,\ldots,j\}$, $\mu_j(i) = 1/j$ for $1\leq i\leq j$, then $\mu_j P = \mu_{j+1}$. This definition is consistent with \eqref{mu_j}; even though $P$ is not reversible, the only part of Proposition \ref{link} requiring reversibility is the proof that the entries of the link $\Lambda$ are nonnegative, which is certainly true if $\beta_j = 0$ for all $1\leq j\leq N-1$.

The mixing time of the chain is order $N$. This shows that the hypothesis of reversibility in Proposition \ref{l2-basic} is necessary: the upper bound would be of order $\log N$. Theorem \ref{main} gives the order $N$ upper bound.

For fixed $0<\beta<1$, let $P_\beta = \beta I + (1-\beta)P$. The relaxation time of $P_\beta$ is $1/(1-\beta)$, and the mixing time is of order $N/(1-\beta) = N\trel$.

\subsection{Lazy random walk on the hypercube}
\label{hypercube}

Fix $n\geq 1$ and let $\X = (\Z / 2\Z)^n = \{(x_1,\ldots,x_n) \st x_i \in \{0,1\} \}$. Define $P$ by $P(x,x) = 1/2$ and $P(x,y) = 1/2n$ if $x,y\in\X$ differ in exactly one coordinate. The stationary distribution is uniform. The absolute spectral gap is $1/n$, so Proposition \ref{l2-basic} says that the mixing time is at least order $n$ and at most order $n^2$. It can be proved by coupling or diagonalization (see e.g. \cite{LPW09}) that the correct order is $n\log n$. Theorem \ref{rev-sharpen} gives an upper bound of $n\cdot 2^n$, which is extremely conservative. Theorem \ref{sst} only improves the bound to $2^n$.


\section{Stochastic interpretation of eigenvalues}

The exposition in this section closely follows that of Fill \cite{F09a} and Miclo \cite{M10}. The main difference is that our result holds for reversible chains with arbitrary stationary distribution, while Fill considers only birth and death chains and Miclo considers reversible chains with an absorbing state. The restrictions allow those authors to conclude that the strong stationary times they construct are optimal, but that does not concern us here.

\begin{prop} \label{link}
Let $P$ be a Markov transition matrix on a state space $\X$ of size $N$. Suppose $P\in\overline{\RRN}$ is reversible with respect to $\pi$ and its eigenvalues $0 \leq \beta_1 \leq \cdots \leq \beta_{N-1} < \beta_N = 1$ are nonnegative. Let $\mu$ be any probability distribution on $\X$, and define the transition matrix $Q$ on $\{1,\ldots,N\}$ by \eqref{pure-birth-matrix}. Then there is a stochastic link $\Lambda$ with rows indexed by $\{1,\ldots,N\}$ and columns indexed by $\X$ such that the first row of $\Lambda$ is $\mu$, the $N$th row of $\Lambda$ is $\pi$, and $\Lambda P = Q \Lambda$.
\end{prop}

\begin{proof}
We need only prove the proposition for $P\in\RRN$. If $P\in\overline{\RRN}$ but not $\RRN$, write $P$ as a limit of matrices in $\RRN$ and take limits throughout the proof.

Set $\Lambda(1,\cdot) = \mu$ and for $2\leq j\leq N$,
\begin{equation} \label{lambda-def}
\Lambda(j,\cdot) = \mu \cdot \prod_{i=1}^{j-1} \frac{P-\beta_i I}{1-\beta_i}.
\end{equation}
Let $\mu_j = \Lambda(j,\cdot)$. These are the ``local equilibria'' of Miclo \cite{M10}. For $1\leq j\leq N-1$, \eqref{lambda-def} implies that
\begin{equation} \label{mu_j}
\mu_j P = \beta_j \mu_j + (1-\beta_j) \mu_{j+1}.
\end{equation}
Equivalently, if $\delta_j$ is the measure on $\{1,\ldots,N\}$ given by $\delta_j(i) = 1$ if $i=j$ and 0 otherwise, we have $\delta_j \Lambda P = \delta_j Q \Lambda$.

When $j=N$, we use that $\mu_N = \pi$. This is proved as follows. Using induction on \eqref{mu_j}, each $\mu_j$ has total mass $\mu_j(\X) = 1$. Multiplying \eqref{lambda-def} on the right by $P-\beta_N I = P-I$ when $j=N$, the Cayley-Hamilton Theorem implies that $\mu_N(P-I) = 0$. Since the eigenvalue 1 has multiplicity 1, the left eigenspace of $P$ corresponding to the eigenvalue 1 is one-dimensional, so $\mu_N$ is a scalar multiple of $\pi$. We have $\mu_N(\X) = 1$, so $\mu_N = \pi$.

Now, $\mu_N P = \mu_N$, so $\delta_N \Lambda P = \delta_N Q \Lambda$. We have proved that $\Lambda P = Q\Lambda$.

All that remains is to prove the entries of $\Lambda$ are nonnegative. This follows as in Lemma 4.1 of \cite{F09a} from a linear algebra result of Micchelli and Willoughby \cite{MW79}. This is the only place we use reversibility (but there are examples of nonreversible $P$ with nonnegative real eigenvalues for which $\Lambda$ has negative entries).
\end{proof}

\begin{proof}[Proof of Theorem \ref{sst}]
Essentially, Theorem \ref{sst} follows from the reasoning in the last paragraph of Section 1.2. Here we explicitly construct the $\tau_j$ in the statement of the theorem.

As in the last paragraph of Section 1.2, let $(Y_t: t\geq 0)$ be the Markov chain on $\{1,\ldots,N\}$ with initial value $Y_0 = 1$ and transition matrix $Q$. For $1\leq j\leq N$, set $T_j = \min\{t\geq 0 \st Y_t = j\}$. Then $T_1 = 0$ and for $j\geq 1$, $T_{j+1}-T_j$ is independent of $T_j$ and distributed geometrically: $\Prob(T_{j+1}-T_j = t) = \beta_j^{t-1}(1-\beta_j)$ for $t\geq 1$. Set $\tau_j = T_{j+1}-T_j$. This definition combined with the reasoning in Section 1.2 proves Theorem \ref{sst}.
\end{proof}

\begin{proof}[Proof of Theorem \ref{rev-sharpen}]
Assume first that the eigenvalues of $P$ are all nonnegative. Applying Theorem \ref{sst} and using that total variation distance is bounded by separation distance, we have that for all $x\in\X$ and $t\geq 0$,
\[
\|P^t(x,\cdot) - \pi\|_\TV \leq \Prob(\tau_1 + \cdots + \tau_{N-1} > t)
\]
where the $\tau_j$ are independent and stochastically dominated by iid geometric random variables $\eta_j$ satisfying $\Prob(\eta_j = t) = \beta_\star^{t-1}(1-\beta_\star)$ for $t\geq 1$. The sum $\eta = \eta_1 + \cdots + \eta_{N-1}$ is a negative binomial random variable. If $Y \sim \Binom(t,1-\beta_\star)$, then $\Prob(\eta > t) = \Prob(Y \leq N-2)$. Theorem \ref{rev-sharpen} follows from applying a standard Chernoff bound to $Y$. For $\delta\geq 0$,
\begin{equation} \label{Chernoff}
\Prob(Y \leq (1-\delta)\cdot (1-\beta_\star)t) \leq \exp\left( \frac{-\delta^2\cdot (1-\beta_\star)t}{2} \right).
\end{equation}
Choose $\delta$ so that $(1-\delta)(1-\beta_\star)t = N-2$. Then the right side of \eqref{Chernoff} is less than or equal to $\e$ exactly when
\begin{equation} \label{quadratic}
[(1-\beta_\star)t]^2 - 2(N-2+\log\e^{-1})[(1-\beta_\star)t] + (N-2)^2 \geq 0.
\end{equation}
The left side of \eqref{quadratic} is a quadratic function of $(1-\beta_\star)t$, so it is positive when $(1-\beta_\star)t$ is greater than the higher root given by the quadratic formula. In other words, when
\begin{equation} \label{quadratic-2}
(1-\beta_\star)t \geq N-2+\log\e^{-1} + \sqrt{2(N-2)\log\e^{-1} + (\log\e^{-1})^2},
\end{equation}
then \eqref{quadratic} holds. In addition, for such $t$ we have $\delta\geq 0$, so \eqref{Chernoff} is valid. Using that $\sqrt{a+b} \leq \sqrt{a}+\sqrt{b}$ for $a,b\geq 0$, we see that if
\begin{equation} \label{final-bound}
t \geq \trel\left[ N-2 + \log\e^{-1} + \sqrt{2(N-2)\log\e^{-1}} + \sqrt{(\log\e^{-1})^2} \right],
\end{equation}
then \eqref{quadratic-2} is true, meaning that
\[
\|P^t(x,\cdot)-\pi\|_\TV \leq \Prob(\eta>t) = \Prob(Y\leq N-2) \leq \exp\left( \frac{-\delta^2\cdot (1-\beta_\star)t}{2} \right) \leq \e.
\]
Hence $\tmix(\e)$ is at most the right side of \eqref{final-bound}, plus 1 because $\tmix(\e)$ must be an integer. This proves Theorem \ref{rev-sharpen} when the eigenvalues of $P$ are nonnegative.

If $P$ has negative eigenvalues, we apply the preceding argument to $P^2$, which has absolute spectral gap $1-\beta_\star^2$. The equivalent of \eqref{final-bound} is that if
\[
t \geq \frac{1}{1-\beta_\star^2} \left[ N-2 + 2\log\e^{-1} + \sqrt{2(N-2)\log\e^{-1}} \right],
\]
then $\|P^{2t}(x,\cdot)-\pi\|_\TV \leq \e$. Thus,
\[
\tmix(\e) \leq 2 \left\lceil \frac{\trel}{1+\beta_\star} \left[ N-2 + 2\log\e^{-1} + \sqrt{2(N-2)\log\e^{-1}} \right] \right\rceil,
\]
which proves Theorem \ref{rev-sharpen} using that $\beta_\star \geq 0$.
\end{proof}

\section{Powers of the companion matrix}

In this section, we show that if $C$ is the companion matrix for a monic polynomial $g(x)$, then the entries of powers of $C$ are certain Schur polynomials in the roots of $g$. This result will be used in Section 5 to prove Theorem \ref{main}.

In previous work, Chen and Louck \cite{CL96} obtained a formula for the entries of powers of $C$ in terms of the coefficients of $g$. That formula is not well-suited for our purposes. Hou and Mu, in a remark at the end of \cite{HM03}, provide a formula that is equivalent to the main result of this section. The proof given here is self-contained and direct.

We will use the definition of Schur polynomials given in terms of semistandard Young tableaux. Fix a positive integer $m$, and let $k_1 \geq k_2 \geq \cdots \geq k_\ell \geq 1$ be a weakly decreasing sequence of positive integers with $\ell \leq m$. A semistandard Young tableau of shape $(k_1,\ldots,k_\ell)$ on the alphabet $\{1,\ldots,m\}$ is a list of tuples $\{(a_{11},a_{12},\ldots,a_{1k_1}); (a_{21},a_{22},\ldots,a_{2k_2}); \ldots; (a_{\ell 1},a_{\ell 2},\ldots,a_{\ell k_\ell})\}$ such that each $a_{ij} \in \{1,\ldots,m\}$; each row $(a_{i1},a_{i2},\ldots,a_{ik_i})$ is a weakly increasing sequence; and each column $(a_{1j},a_{2j},\ldots,a_{\ell_j j})$ is a strictly increasing sequence. Here $\ell_j = \max\{1\leq i\leq\ell \st j\leq k_i\}$. The weight $w(T)$ of a semistandard Young tableau $T$ on the alphabet $\{1,\ldots,m\}$ is the $m$-tuple $(w_1,\ldots,w_m)$, where each $w_k = \#\{(i,j) \st a_{ij} = k\}$.

Given the positive integer $m$ and the weakly decreasing sequence $(k_1,\ldots,k_\ell)$, define the Schur polynomial
\[
s_{(k_1,\ldots,k_\ell)}(x_1,\ldots,x_m) = \sum_{\substack{\text{semistandard Young tableaux $T$} \\ \text{of shape $(k_1,\ldots,k_\ell)$} \\ \text{on the alphabet $\{1,\ldots,m\}$}}} x^{w(T)},
\]
where if $w = (w_1,\ldots,w_m)$, then $x^w$ is shorthand notation for $x_1^{w_1}\cdots x_m^{w_m}$.

As an example, there are three semistandard Young tableaux of shape $(2,2,1)$ on the alphabet $\{1,2,3\}$. They are: $\{(1,1); (2,2); (3)\}$, $\{(1,1); (2,3); (3)\}$, and $\{(1,2); (2,3); (3)\}$. They can be graphically represented as follows.
\[
\young(11,22,3) \qquad \young(11,23,3) \qquad \young(12,23,3)
\]
The associated Schur polynomial is
\[
s_{(2,2,1)}(x_1,x_2,x_3) = x_1^2 x_2^2 x_3 + x_1^2 x_2 x_3^2 + x_1 x_2^2 x_3^2.
\]
It is well-known that all Schur polynomials are symmetric, and as such, they can be written in terms of the elementary symmetric polynomials. For $1\leq k\leq m$, the elementary symmetric polynomial $e_k$ in $m$ variables is given by
\[
e_k = e_k(x_1,\ldots,x_m) = \sum_{1\leq i_1 < \cdots < i_k \leq m} x_{i_1} \cdots x_{i_k}.
\]
In the above example,
\[
s_{(2,2,1)}(x_1,x_2,x_3) = (x_1x_2 + x_1x_3 + x_2x_3)x_1x_2x_3 = e_2 e_3.
\]

Let $g(x)$ be a monic polynomial of degree $m$ with roots $x_1,\ldots,x_m$. Then the coefficients of $g$ are elementary symmetric polynomials in the roots $x_1,\ldots,x_m$: $g(x) = x^m - e_1x^{m-1} + e_2x^{m-2} - \cdots + (-1)^m e_m$. The companion matrix of $g$ is given by
\[
C = 
\begin{bmatrix}
0 & 0 & \cdots & \cdots & 0 & (-1)^{m-1}e_m \\
1 & 0 & 0 & \cdots & 0 & (-1)^{m-2}e_{m-1} \\
0 & 1 & 0 & \cdots & 0 & (-1)^{m-3}e_{m-2} \\
0 & 0 & 1 & \cdots & 0 & (-1)^{m-4}e_{m-3} \\
\vdots & \vdots & \vdots & \ddots & \vdots & \vdots \\
0 & 0 & 0 & \cdots & 1 & e_1
\end{bmatrix}.
\]
For $1\leq j\leq m-1$, $C(i,j) = 1$ if $i = j+1$ and 0 otherwise. In the last column, $C(i,m) = (-1)^{m-i}e_{m-i+1}$.

We can now state the main theorem of this section.
\begin{thm} \label{companion}
Let $g(x)$ be a monic polynomial of degree $m$, and let $C$ be its companion matrix. Suppose the roots of $g$ are $x_1,\ldots,x_m$. For $t\geq 0$, if $t+j-m \geq 1$, then
\[
C^t(i,j) = (-1)^{m-i} s_{(t+j-m,\underbrace{\scriptstyle 1,1,\ldots,1}_{m-i})}(x_1,\ldots,x_m).
\]
(The notation $(t+j-m,\underbrace{1,1,\ldots,1}_{m-i})$ means $t+j-m$ followed by a sequence of $m-i$ ones.) If $t+j-m < 1$, then
\[
C^t(i,j) = \begin{cases} 1 & \text{if } i = t+j, \\ 0 & \text{otherwise.} \end{cases}
\]
\end{thm}

The proof of Theorem \ref{companion} will use the following lemma, which is a special case of Pi\`eri's Rule.
\begin{lemma} \label{schur-induction}
Fix $m\geq 1$. For any $k\geq 1$ and any $1\leq\ell\leq m$,
\[
s_{(k,\underbrace{\scriptstyle 1,1,\ldots,1}_\ell)} + s_{(k+1,\underbrace{\scriptstyle 1,1,\ldots,1}_{\ell-1})} = s_{(k)} \cdot e_\ell.
\]
Here, each Schur polynomial and elementary symmetric polynomial is in the variables $x_1,\ldots,x_m$. When $\ell=m$, the first term of the left hand side is taken to be zero.
\end{lemma}

\begin{proof}
We have
\begin{align*}
s_{(k,\underbrace{\scriptstyle 1,1,\ldots,1}_\ell)} &= \sum_{\substack{1\leq a_1 \leq a_2 \cdots \leq a_k \leq m \\ 1\leq b_1 < b_2 < \cdots < b_\ell \leq m \\ a_1 < b_1}} x_{a_1} \cdots x_{a_k} \cdot x_{b_1} \cdots x_{b_\ell}, \\
s_{(k+1,\underbrace{\scriptstyle 1,1,\ldots,1}_{\ell-1})} &= \sum_{\substack{1\leq a_1 \leq a_2 \cdots \leq a_k \leq m \\ 1\leq b_1 < b_2 < \cdots < b_\ell \leq m \\ a_1 \geq b_1}} x_{a_1} \cdots x_{a_k} \cdot x_{b_1} \cdots x_{b_\ell}.
\end{align*}
Meanwhile,
\[
s_{(k)} \cdot e_\ell = \sum_{\substack{1\leq a_1 \leq a_2 \cdots \leq a_k \leq m \\ 1\leq b_1 < b_2 < \cdots < b_\ell \leq m}} x_{a_1} \cdots x_{a_k} \cdot x_{b_1} \cdots x_{b_\ell}.
\]
Thus, the left and right hand sides are equal.
\end{proof}

\begin{proof}[Proof of Theorem \ref{companion}]
The proof is by induction on $t$. For $t=0$, the theorem is trivially true. Suppose the theorem is true for $t$. For all $1\leq i,j \leq m$,
\[
C^{t+1}(i,j) = \sum_{k=1}^m C(i,k)C^t(k,j) = C^t(i-1,j) + (-1)^{m-i}e_{m-i+1} C^t(m,j),
\]
where we take $C^t(i-1,j) = 0$ when $i=1$.

We split into three cases depending on the sign of $t+j-m$. If $t+j-m \geq 1$, then
\begin{align*}
C^t(i-1,j) &= (-1)^{m-i+1} s_{(t+j-m,\underbrace{\scriptstyle 1,1,\ldots,1}_{m-i+1})}, \\
C^t(m,j) &= s_{(t+j-m)}.
\end{align*}
Hence
\[
C^{t+1}(i,j) = (-1)^{m-i+1} s_{(t+j-m,\underbrace{\scriptstyle 1,1,\ldots,1}_{m-i+1})} + (-1)^{m-i}e_{m-i+1} s_{(t+j-m)}.
\]
Applying Lemma \ref{schur-induction} with $k = t+j-m$ and $\ell = m-i+1$, we conclude that
\[
C^{t+1}(i,j) = (-1)^{m-i} s_{(t+j-m+1,\underbrace{\scriptstyle 1,1,\ldots,1}_{m-i})},
\]
which completes the induction.

If $t+j-m = 0$, then $C^t(i-1,j) = 0$ and $C^t(m,j) = 1$, so
\[
C^{t+1}(i,j) = (-1)^{m-i}e_{m-i+1} = (-1)^{m-i} s_{(t+j-m+1,\underbrace{\scriptstyle 1,1,\ldots,1}_{m-i})},
\]
as desired. If $t+j-m \leq -1$, then $C^t(m,j) = 0$, so
\[
C^{t+1}(i,j) = C^t(i-1,j) = \begin{cases} 1 & \text{if } i-1 = t+j, \\ 0 & \text{otherwise} \end{cases} = \begin{cases} 1 & \text{if } i = t+1+j, \\ 0 & \text{otherwise,} \end{cases}
\]
as desired. This completes the proof.
\end{proof}

\section{Proof of Theorem \ref{main}}

The first step in proving Theorem \ref{main} is establishing the recurrence relation for the sequence $(P^t - \Pi \st t\geq 0)$.

\begin{prop} \label{cayley-hamilton}
Let $P$ be a Markov transition matrix on a state space $\X$ of size $N$. Suppose the eigenvalues $\beta_1,\ldots,\beta_N = 1$ satisfy $\beta_j\neq 1$ for $j<N$. Denote by $\pi$ the stationary distribution for $P$, and define the $N\times N$ matrix $\Pi$ by $\Pi(x,y) = \pi(y)$ for all $x,y\in\X$. Let $e_1,\ldots,e_{N-1}$ be the elementary symmetric polynomials in the eigenvalues $\beta_1,\ldots,\beta_{N-1}$, and let
\[
g(x) = (x-\beta_1)\cdots(x-\beta_{N-1}) = x^{N-1} - e_1 x^{N-2} + e_2 x^{N-3} - \cdots + (-1)^{N-1}e_{N-1}.
\]
Then
\[
(P^{N-1} - \Pi) - e_1(P^{N-2} - \Pi) + e_2(P^{N-3} - \Pi) - \cdots + (-1)^{N-1}e_{N-1}(I - \Pi) = 0.
\]
\end{prop}

\begin{proof}
First note that for all $t\geq 1$, $(P - \Pi)^t = P^t - \Pi$. This can be proved by induction on $t$. For $t=1$, it is certainly true. If it holds for $t$, then
\[
\begin{split}
(P - \Pi)^{t+1} = (P - \Pi)(P^t - \Pi) &= P^{t+1} - \Pi P^t - P^t \Pi + \Pi^2 \\
&= P^{t+1} - \Pi - \Pi + \Pi = P^{t+1} - \Pi,
\end{split}
\]
completing the proof.

We will show that $g(P - \Pi) = (-1)^{N-1}e_{N-1}\Pi$, which is equivalent to the desired statement. Suppose that $v$ is a column vector with $\pi v = 0$. The characteristic polynomial of $P$ is $(x-1)g(x)$, so $(P-I)g(P) = 0$ by the Cayley-Hamilton Theorem. Hence $g(P)v$ is in the kernel of $P-I$, so it must be a constant multiple of the eigenvector $\1$ of all ones. Write $g(P)v = c\1$. Then $\pi g(P)v = c$. But also,
\[
\pi g(P) = (1-\beta_1)\cdots(1-\beta_{N-1})\pi,
\]
so $\pi g(P)v = (1-\beta_1)\cdots(1-\beta_{N-1})\pi v = 0$. Thus, $c=0$ and $g(P)v = 0$. Further,
\[
\begin{split}
g(P - \Pi)v &= (P^{N-1} - \Pi)v - e_1(P^{N-2} - \Pi)v + \cdots + (-1)^{N-1}e_{N-1}v \\
&= P^{N-1}v - e_1 P^{N-2}v + \cdots + (-1)^{N-1}e_{N-1}v \\
&= g(P)v = 0.
\end{split}
\]

Any column vector $w$ can be written as $w = [w - (\pi w)\1] + (\pi w)\1$, where $\pi[w - (\pi w)\1] = 0$. So,
\[
\begin{split}
g(P - \Pi)w &= g(P - \Pi)(\pi w)\1 \\
&= (\pi w)(P - \Pi - \beta_1 I)\cdots(P - \Pi - \beta_{N-1} I)\1 \\
&= (\pi w)(-\beta_1)\cdots(-\beta_{N-1})\1 \\
&= (-1)^{N-1}e_{N-1}\1(\pi w).
\end{split}
\]
It follows that $g(P - \Pi) = (-1)^{N-1}e_{N-1}\Pi$, as desired.
\end{proof}

Proposition \ref{cayley-hamilton} expresses $P^{N-1} - \Pi$ as a linear combination of the matrices $P^k - \Pi$ for $0\leq k\leq N-2$. This is a linear recurrence relation, so any power $P^t - \Pi$ can be written as a linear combination of the same $N-1$ matrices. We are interested in the coefficients in that linear combination.

\begin{prop} \label{recurrence}
Under the conditions of Proposition \ref{cayley-hamilton}, let $C$ be the companion matrix of $g(x)$. Then for all $t\geq 0$,
\[
P^t - \Pi = \sum_{k=0}^{N-2} C^t(k+1,1)(P^k - \Pi).
\]
\end{prop}

\begin{proof}
The proof is by induction on $t$. When $t=0$, the statement is trivial. For the inductive step, suppose the statement is true for $t$. Then
\[
\begin{split}
P^{t+1} - \Pi &= \sum_{k=0}^{N-2} C^t(k+1,1)(P^{k+1} - \Pi) \\
&= \sum_{k=0}^{N-2} C^t(k,1)(P^k - \Pi) + C^t(N-1,1)(P^{N-1} - \Pi),
\end{split}
\]
where we take $C^t(k,1) = 0$ when $k=0$. By Proposition \ref{cayley-hamilton},
\[
P^{N-1} - \Pi = e_1(P^{N-2} - \Pi) - e_2(P^{N-3} - \Pi) + \cdots + (-1)^{N-2}e_{N-1}(I - \Pi),
\]
meaning that
\[
P^{t+1} - \Pi = \sum_{k=0}^{N-2} [C^t(k,1) + (-1)^{N-2-k}e_{N-1-k}C^t(N-1,1)](P^k - \Pi).
\]
By the definition of $C$,
\[
\begin{split}
C^{t+1}(k+1,1) &= \sum_{j=1}^{N-1} C(k+1,j)C^t(j,1)  \\
&= C^t(k,1) + (-1)^{N-2-k}e_{N-1-k}C^t(N-1,1),
\end{split}
\]
which proves that
\[
P^{t+1} - \Pi = \sum_{k=0}^{N-2} C^{t+1}(k+1,1)(P^k - \Pi). \qedhere
\]
\end{proof}

To bound the total variation distance after $t$ steps, we need only bound the coefficients $C^t(k+1,1)$. This is accomplished using Theorem \ref{companion}.

\begin{thm} \label{tv-bound}
Let $P$ be a Markov transition matrix on a state space $\X$ of size $N$. Suppose the eigenvalues $\beta_1,\ldots,\beta_N = 1$ satisfy $|\beta_j| \leq \beta_\star < 1$ for $j<N$. Denote by $\pi$ the stationary distribution for $P$. For all $x\in\X$ and all $t\geq N-1$,
\[
\|P^t(x,\cdot) - \pi\|_\TV \leq (N-1)\binom{t}{N-1} \sum_{k=0}^{N-2} \binom{N-2}{k} \frac{\beta_\star^{t-k}}{t-k}.
\]
\end{thm}

\begin{proof}
By Proposition \ref{recurrence},
\[
\|P^t(x,\cdot) - \pi\|_\TV \leq \sum_{k=0}^{N-2} |C^t(k+1,1)| \cdot \|P^k(x,\cdot) - \pi\|_\TV \leq \sum_{k=0}^{N-2} |C^t(k+1,1)|.
\]
Theorem \ref{companion} gives
\[
|C^t(k+1,1)| = |s_{(t+2-N,\underbrace{\scriptstyle 1,1,\ldots,1}_{N-2-k})}(\beta_1,\ldots,\beta_{N-1})| \leq s_{(t+2-N,\underbrace{\scriptstyle 1,1,\ldots,1}_{N-2-k})}(\beta_\star,\ldots,\beta_\star).
\]
This last quantity is simply $\beta_\star^{(t+2-N)+(N-2-k)}$ multiplied by the number of semistandard Young tableaux of shape $(t+2-N,\underbrace{1,1,\ldots,1}_{N-2-k})$ on the alphabet $\{1,\ldots,N-1\}$. By Lemma \ref{ssyt-count}, proved below, the number of such tableaux is
\[
\binom{t-k-1}{t+1-N} \binom{t}{t-k} = (N-1)\binom{t}{N-1} \cdot \binom{N-2}{k}\frac{1}{t-k}.
\]
Thus,
\[
\|P^t(x,\cdot) - \pi\|_\TV \leq \sum_{k=0}^{N-2} (N-1)\binom{t}{N-1} \cdot \binom{N-2}{k}\frac{\beta_\star^{t-k}}{t-k}. \qedhere
\]
\end{proof}

\begin{lemma} \label{ssyt-count}
Let $m\geq 1$, $b\geq 1$, and $c\geq 0$. The number of semistandard Young tableaux of shape $(b,\underbrace{1,1,\ldots,1}_c)$ on the alphabet $\{1,\ldots,m\}$ is
\[
\binom{b+c-1}{b-1} \binom{b+m-1}{b+c}.
\]
\end{lemma}

\begin{proof}
Fix $1\leq i\leq m$. We count the number of semistandard Young tableaux of shape $(b,\underbrace{1,1,\ldots,1}_c)$ whose top left element is $a_{11} = i$. The rest of the first row must be a weakly increasing sequence of length $b-1$ whose first element is at least $i$ and whose last element is at most $m$. The standard ``stars and bars'' argument shows that there are $\binom{m-i+b-1}{m-i}$ ways to fill in the first row.

The rest of the first column must be a strictly increasing sequence of length $c$ whose first element is at least $i+1$ and whose last element is at most $m$. There are $\binom{m-i}{c}$ ways to fill in the first column. Therefore, the total number of semistandard Young tableaux of shape $(b,\underbrace{1,1,\ldots,1}_c)$ on the alphabet $\{1,\ldots,m\}$ is
\[
\sum_{i=1}^m \binom{m-i+b-1}{m-i} \binom{m-i}{c} = \sum_{i=1}^m \binom{b+c-1}{b-1} \binom{m-i+b-1}{b+c-1}.
\]
We have
\[
\sum_{i=1}^m \binom{m-i+b-1}{b+c-1} = \sum_{j=c}^{m-1} \binom{b+j-1}{b+c-1} = \binom{b+m-1}{b+c},
\]
which completes the proof.
\end{proof}

Theorem \ref{tv-bound} bounds the total variation distance from stationarity after $t$ steps for $t\geq N-1$. Note that Example \ref{pure-birth-section} has $\|P^t(1,\cdot) - \pi\|_\TV = 1$ for $t\leq N-2$, so $t=N-1$ is the first time that any nontrivial upper bound is possible. On the other hand, the bound given by Theorem \ref{tv-bound} is far from optimal for $t$ near $N$. When $t=N-1$, the bound simplifies to
\[
\|P^t(x,\cdot) - \pi\|_\TV \leq (1+\beta_\star)^{N-1} - 1,
\]
which is much larger than 1 unless $\beta_\star$ is very close to zero.

One interpretation of Theorem \ref{main} is that the bound in Theorem \ref{tv-bound} gets small when $t$ is a constant multiple of $N\trel\log\trel$. Indeed, we are now ready to prove Theorem \ref{main}.

\begin{proof}[Proof of Theorem \ref{main}]
As usual, denote the state space by $\X$ and the eigenvalues by $\beta_1,\ldots,\beta_N = 1$, with $\beta_\star = \max\{ |\beta_j| \st 1\leq j\leq N-1 \}$. If $N=1$, the theorem is trivially true, so assume $N\geq 2$. If $\beta_\star=1$, then $\trel = \infty$ and the theorem has no content. If $\beta_\star < 1$, then in particular, there is a unique stationary distribution $\pi$. Theorem \ref{tv-bound} says that for all $x\in\X$ and $t\geq N-1$,
\[
\begin{split}
\|P^t(x,\cdot) - \pi\|_\TV &\leq (N-1)\binom{t}{N-1} \sum_{k=0}^{N-2} \binom{N-2}{k} \frac{\beta_\star^{t-k}}{t-k} \\
&\leq \binom{t}{N-2} \beta_\star^{t-(N-2)} \sum_{k=0}^{N-2} \binom{N-2}{k} \beta_\star^{N-2-k} \\
&= \binom{t}{N-2} \beta_\star^{t-(N-2)} (1+\beta_\star)^{N-2} \\
&\leq 2^{N-2} \frac{t^{N-2}}{(N-2)!} \beta_\star^{t-(N-2)}.
\end{split}
\]
If $\beta_\star = 0$, then $\|P^t(x,\cdot) - \pi\|_\TV = 0$ for $t\geq N-1$, and Theorem \ref{main} is proved so long as the right hand side $t_\ast = 2\trel \left[ \log\e^{-1} - 1 + 2(1+\log 2)N + N\log\trel \right]$ is greater than or equal to $N-1$. This is true: since $\trel \geq 1$,
\begin{equation} \label{t-ast-n}
t_\ast - (N-1) \geq -1 + (3 + 4\log 2)N \geq 2 + 4\log 2.
\end{equation}

If $\beta_\star > 0$, then given $0<\e<1$, it will suffice to find an integer $N-1 \leq t \leq t_\ast$ such that
\begin{equation} \label{desired-ineq}
2^{N-2} \frac{t^{N-2}}{(N-2)!} \beta_\star^{t-(N-2)} \leq \e.
\end{equation}
For convenience, let $d = N-2$. As $\log d! \geq d\log d - d + 1$, if the following inequality is true, \eqref{desired-ineq} will also hold:
\begin{equation} \label{desired-2}
d\log 2 + d\log t - d\log d + d - 1 + (t-d)\log\beta_\star \leq \log\e.
\end{equation}
Let $f(t) = d\log t + (t-d)\log\beta_\star$. Then $f'(t) = d/t + \log\beta_\star$ and $f''(t) = -d/t^2$. So $f$ is a concave function and will lie below its linearization at any point $t_0$. Choose $t_0 = 2d/\log\beta_\star^{-1}$. Then
\[
f(t) \leq f(t_0) + f'(t_0)(t-t_0) = d\log\left( \frac{2d}{\log\beta_\star^{-1}} \right) - \frac{1}{2}\log\beta_\star^{-1} \left(t - \frac{2d}{\log\beta_\star^{-1}} \right).
\]
If
\begin{equation} \label{t-lower}
t \geq \frac{2}{\log\beta_\star^{-1}} \left[ \log\e^{-1} - 1 + 2(1+\log 2)d + d\log\left( \frac{1}{\log\beta_\star^{-1}} \right) \right],
\end{equation}
then
\[
\begin{split}
f(t) &\leq d\log\left( \frac{2d}{\log\beta_\star^{-1}} \right) - \frac{1}{2}\log\beta_\star^{-1} \left(t - \frac{2d}{\log\beta_\star^{-1}} \right) \\
&\leq -d\log 2 + d\log d - d - \log\e^{-1} + 1,
\end{split}
\]
which yields \eqref{desired-2}. Because $\trel \geq 1/\log\beta_\star^{-1}$, any $t$ such that
\[
t\geq 2\trel \left[ \log\e^{-1} - 1 + 2(1+\log 2)d + d\log\trel \right]
\]
will satisfy \eqref{desired-ineq}.

Since $\trel\geq 1$, we have
\[
t_\ast - 2\trel \left[ \log\e^{-1} - 1 + 2(1+\log 2)d + d\log\trel \right] \geq 8(1 + \log 2) > 1.
\]
Along with \eqref{t-ast-n}, this shows that $\lfloor t_\ast \rfloor$ is an integer greater than $N-1$ satisfying \eqref{desired-ineq}. Thus $\tmix(\e) \leq t_\ast$, and the proof is finished.

Note that the linearization argument does not give away too much. If $\e$ is fixed, the leading term in the lower bound \eqref{t-lower} is
\[
\frac{2d}{\log\beta_\star^{-1}} \log\left( \frac{1}{\log\beta_\star^{-1}} \right).
\]
On the other hand, if we were to set
\[
t = \frac{d}{\log\beta_\star^{-1}} \log\left( \frac{1}{\log\beta_\star^{-1}} \right),
\]
then the left hand side of \eqref{desired-2} would tend to $\infty$ as $\beta_\star\to 1$.
\end{proof}

\bibliographystyle{amsalpha}

\bibliography{lyapunov}


\end{document}